\documentclass[12pt,oneside]{amsart}

\usepackage{setspace,amssymb,amsmath,mathrsfs,enumitem,textcomp,amsbsy,url}
\usepackage[all]{xy}

\usepackage{datetime}
\renewcommand{\dateseparator}{-}
\renewcommand{\today}{\the\year \dateseparator \twodigit\month
\dateseparator \twodigit\day}

\usepackage
[pdfauthor={Stergios M. Antonakoudis},
 bookmarks=false,
 hyperfootnotes=false]
{hyperref}

\usepackage{needspace,graphicx,extpfeil,pdfpages}

\usepackage{fancyhdr}
\fancypagestyle{plain}
{
  \fancyhf{}               
  \fancyhead[C]{\thepage}  
}

\newcounter{restatecount}
\newcommand{\restate}[3][Theorem]
{
\addtocounter{restatecount}{1}
\theoremstyle{plain}
\newtheorem*{thm-\arabic{restatecount}}{#1 \ref{#2}}
\begin{thm-\arabic{restatecount}}
#3
\end{thm-\arabic{restatecount}}
}

\theoremstyle{plain}
\newtheorem{thm}{Theorem}[section]

\newtheorem{thm-ref}{Theorem}

\newtheorem{cor}[thm]{Corollary}
\newtheorem{prop}[thm]{Proposition}

\theoremstyle{definition}

\newtheorem*{examples}{Examples}

\theoremstyle{remark}
\newtheorem*{remark}{Remark}
\newtheorem*{remarks}{Remarks}

\numberwithin{equation}{section}

\addtocontents{toc}{\protect\setcounter{tocdepth}{1}}

\newcommand{\M}{{\mathcal M}}
\newcommand{\T}{{\mathcal T}}

\newcommand{\C}{{\mathbb C}}
\newcommand{\R}{{\mathbb R}}

\title{Teichm\"uller spaces and bounded symmetric domains do not mix
  isometrically}
 \author{Stergios M. Antonakoudis}

\begin{document}

  \maketitle
   \begin{abstract}
    This paper shows that, in dimensions two or more, there are no
    holomorphic isometries between Teichm\"uller spaces and bounded
    symmetric domains in their intrinsic Kobayashi metric.
   \end{abstract}
  \tableofcontents


\section{Introduction}\label{sec:intro}

We study holomorphic maps between Teichm\"uller spaces
$\T_{g,n} \subset \C^{3g-3+n}$ and bounded symmetric domains
$\mathcal{B}\subset\C^{N}$ in their intrinsic Kobayashi metric. The
main result in this paper is the following theorem.

\begin{thm}\label{thm:bsds:intro}
  Let $\mathcal{B}$ be a bounded symmetric domain and $\T_{g,n}$ be a
  Teichm\"uller space with $\text{dim}_{\C}\mathcal{B},\text{dim}_{\C}\T_{g,n}
  \geq 2$. There are no holomorphic isometric immersions \[\mathcal{B}
  \xhookrightarrow{~~~f~~~} \T_{g,n} \quad \text{or} \quad \T_{g,n}
  \xhookrightarrow{~~~f~~~} \mathcal{B}\] such that $df$ is an isometry for the
  Kobayashi norms on tangent spaces.
\end{thm}

The proof involves ideas from geometric topology and leverages the
description of Teichm\"uller geodesics in terms of measured foliations
and extremal length on Riemann surfaces.

We note the following immediate corollary.
\begin{cor}\label{cor-symmetric}
  There is no locally symmetric variety $\mathcal{V}$ isometrically
  immersed in the moduli space of curves $\M_{g,n}$, nor is there an
  isometric copy of $\M_{g,n}$ in $\mathcal{V}$, for the Kobayashi
  metrics, so long as both have dimension two or more.
\end{cor}

A feature that Teichm\"uller spaces and bounded symmetric domains have
in common is that they contain holomorphic isometric copies of
$\mathbb{CH}^1$ through every point and complex direction; in
particular, in complex dimension one, Teichm\"uller spaces and bounded
symmetric domains coincide.


In higher dimensions, it is known that there are many holomorphic
isometries between Teichm\"uller spaces
$f: \T_{g,n} \hookrightarrow \T_{h,m}$ ~\cite{Kra:survey} and bounded
symmetric domains
$f: \mathcal{B} \hookrightarrow \widetilde{\mathcal{B}}$
~\cite{Helgason:book:dglgss}, respectively, in their intrinsic
Kobayashi metric.

Informally, our results show that in dimensions two or more
Teichm\"uller spaces and bounded symmetric domains \textit{do not mix}
isometrically.~\\


As an application of Theorem~\ref{thm:bsds:intro}, we prove:
\begin{thm}\label{thm:kaehler}
  Let $(\mathcal{M},g)$ be a complete K\"ahler manifold with
  $\text{dim}_{\C}\mathcal{M} \geq 2$ and holomorphic sectional
  curvature at least $-4$. There is no holomorphic map
  $f:\mathcal{M}\rightarrow \T_{g,n}$ such that $df$ is an isometry on
  tangent spaces.
\end{thm}
\begin{proof}
  The monotonicity of holomorphic sectional curvature under
  holomorphic maps and the existence of (totally geodesic) holomorphic
  isometries $\mathbb{CH}^1 \hookrightarrow \T_{g,n}$ through every
  complex direction imply that $\mathcal{M}$ has constant holomorphic
  curvature -4.~\cite{Royden:metric} Since $\mathcal{M}$ is a complete
  K\"ahler manifold, we have $\mathcal{M}\cong\mathbb{CH}^{N}$, which
  is impossible when $N \geq 2$ by Theorem~\ref{thm:bsds:intro}.
\end{proof}
The following corollary is immediate.
\begin{cor}\label{cor:kaehler}
  There is no holomorphic, totally geodesic isometry from a K\"ahler
  manifold $\mathcal{M}$ into a Teichm\"uller space $\T_{g,n}$, so long
  as $\mathcal{M}$ has dimension two or more.
\end{cor}
\subsection*{Questions} We conclude with two open questions.

\noindent\textbf{1.} Is there a holomorphic map $f: (\mathcal{M},g) \rightarrow
\T_{g,n}$ from a Hermitian manifold with $\text{dim}_{\C}\mathcal{M} \geq 2$ such that
$df$ is an isometry on tangent spaces?

\noindent\textbf{2.} Is there a \textit{round} complex two-dimensional
linear slice in $T_X\T_{g,n}$?

\noindent\textit{Theorems \ref{thm:bsds:intro} and \ref{thm:kaehler}
  suggest that the answers to both
  questions are negative.}~\\

\subsection*{Notes and References}~\\
For an introduction to Teichm\"uller spaces, we refer to
~\cite{Gardiner:Lakic:book} and ~\cite{Hubbard:book:T1}; for an
introduction to symmetric spaces and their intrinsic Kobayashi metric,
we refer to ~\cite{Helgason:book:dglgss},~\cite{Satake:book:symmetric}
and ~\cite{Kobayashi:book:hyperbolic}, respectively. We note that the
Kobayashi metric of a bounded symmetric domain $\mathcal{B}$ does not
coincide with its Hermitian symmetric metric, unless it has rank one
ie. $\mathcal{B}\cong \mathbb{CH}^{N}$.

In his pioneering paper~\cite{Royden:metric}, H. L. Royden showed that
the Kobayashi metric of $\T_{g,n}$ coincides with its classical
Teichm\"uller metric and, using this result, he proved that, when
$\text{dim}_{\C}\T_{g}\geq 2$, the group of holomorphic automorphisms
$\text{Aut}(\T_{g})$ is discrete; hence, in particular, $\T_{g,n}$ is
not a symmetric domain. A proof that $\text{Aut}(\T_{g,n})$ is
discrete for all finite-dimensional Teichm\"uller spaces of dimension
two or more is given in~\cite{Earle:Markovic:isometries}.

The existence of isometrically immersed curves, known as Teichm\"uller
curves, in $\M_{g,n}$ has far-reaching applications in the dynamics of
billiards in rational polygons.~\cite{Veech:triangles},~\cite{McMullen:bild}
Corollary~\ref{cor-symmetric} shows that there are no higher
dimensional, locally symmetric, analogues of Teichm\"uller curves.

\subsection*{Acknowledgments}
I wish to thank my thesis advisor, Curtis T. McMullen, for many
insightful discussions which set off the train of thought leading to
the main result in this paper.




\addtocontents{toc}{\protect\setcounter{tocdepth}{1}}

\section{Preliminaries}\label{sec:prelim}
Let $\T_{g,n}$ denote the Teichm\"uller space of marked Riemann
surfaces of genus $g$ with $n$ punctures; it is the orbifold universal
cover of the moduli space of curves $\M_{g,n}$ and it is naturally a
complex manifold of dimension $3g-3+n$. It is known that Teichm\"uller
space can be realized as a (contractible) bounded domain
$\T_{g,n} \subset\C^{3g-3+n}$, by the Bers embeddings.~\cite{Bers:ts:survey}

Let $\mathcal{B}\subset\C^{N}$ be a bounded domain; we call
$\mathcal{B}$ a bounded symmetric domain if every point
$p\in \mathcal{B}$ is an \textit{isolated} fixed point of a
holomorphic involution
$\sigma_{p} : \mathcal{B}\rightarrow \mathcal{B}$, with
$\sigma_{p}^2=\text{id}_{\mathcal{B}}$. Bounded symmetric domains are
contractible and homogeneous as complex manifolds. It is classically
known that all Hermitian symmetric spaces of non-compact type can be
realized as bounded symmetric domains $\mathcal{B}\subset\C^{N}$, by
the Harish-Chandra embeddings.~\cite{Helgason:book:dglgss}

The unit disk $\Delta\cong\{~z\in \C : |z| < 1~\} $ is a bounded
symmetric domain; in fact, it is the unique (up to isomorphism)
contractible bounded domain of complex dimension one. We denote by
$\mathbb{CH}^1$ the unit disk equipped with its Poincar\'e metric
$|dz|/(1-|z|^2)$ of constant curvature $-4$, which we will refer to as
the complex hyperbolic line. Schwarz lemma shows that every
holomorphic map $f:\mathbb{CH}^1\rightarrow \mathbb{CH}^1$ is
non-expanding.
\subsection*{The Kobayashi metric}\label{sec:kobayashi-metric}~\cite{Kobayashi:book:hyperbolic}
Let $\mathcal{B}\subset\C^{N}$ be a bounded domain, its intrinsic
Kobayashi metric is the \textit{largest} complex Finsler metric such
that every holomorphic map $f: \mathbb{CH}^1 \rightarrow \mathcal{B}$
is non-expanding: $||f'(0)||_{\mathcal{B}}\leq 1$. It determines both
a family of norms $||\cdot||_{\mathcal{B}}$ on the tangent bundle
$T\mathcal{B}$ and a distance $d_{\mathcal{B}}(\cdot,\cdot)$ on pairs
of points.

The Kobayashi metric has the fundamental property that every
holomorphic map between complex domains is non-expanding; in
particular, every holomorphic automorphism is an isometry. The
Kobayashi metric of complex domain depends only on its structure as a
complex manifold.

\begin{examples}~\\
  1. $\mathbb{CH}^1$ realises the unit disk $\Delta$ with its
  Kobayashi metric. The Kobayashi metric on the unit ball
  $\mathbb{CH}^2\cong\{~~(z,w)~~ | ~~ |z|^2 + |w|^2 < 1~~ \} \subset
  \C^2$ coincides with its unique (complete) invariant Ka\"ehler metric
  of constant holomorphic curvature -4.~\\ 
  2. The Kobayashi metric on the bi-disk $\mathbb{CH}^1\times\mathbb{CH}^1$ 
  coincides with the sup-metric of the two factors. It is a complex Finsler metric; 
  it is not a Hermitian metric.~\\
  3. The Kobayashi metric on $\T_{g,n}$ coincides with the classical
  Teichm\"uller metric, which endows $\T_{g,n}$ with the structure of
  a complete geodesic metric space.
\end{examples}
Incidentally, examples 1 and 2 above describe all bounded symmetric
domains up to isomorphism in complex dimensions one and two. We
discuss example 3 in more detail below.~\\
\subsection*{Teichm\"uller space}~\cite{Gardiner:Lakic:book},~\cite{Hubbard:book:T1}
Let $\Sigma_{g,n}$ be a connected, oriented surface of genus $g$ and
$n$ punctures and $\T_{g,n}$ denote the Teichm\"uller space of Riemann
surfaces marked by $\Sigma_{g,n}$. A point in $\T_{g,n}$ is specified
by an orientation preserving homeorphism
$\phi: \Sigma_{g,n} \rightarrow X$ to a Riemann surface of finite
type, up to a natural equivalence relation\footnote{Two marked Riemann
  surfaces $ \phi: \Sigma_{g,n} \rightarrow X$,
  $\psi: \Sigma_{g,n}\rightarrow Y$ are equivalent if
  $\psi \circ {\phi}^{-1}: X \rightarrow Y$ is isotopic to a
  holomorphic bijection.}.

Teichm\"uller space $\T_{g,n}$ is naturally a complex manifold of
dimension $3g-3+n$ and forgetting the marking realises $\T_{g,n}$ as
the complex orbifold universal cover of the moduli space $\M_{g,n}$.
When it is clear from the context we often denote a point specified by
$\phi: \Sigma_{g,n} \rightarrow X$ simply by $X$.~\\

\subsection*{Quadratic differentials}
For each $X \in \T_{g,n}$, we let $Q(X)$ denote the space of
holomorphic quadratic differentials $q=q(z)(dz)^2$ on $X$ with finite
total mass: $ ||q||_{1} = \int_{X} |q(z)||dz|^2 < +\infty$, which
means that $q$ has at worse simple poles at the punctures of $X$.

The tangent and cotangent spaces to Teichm\"uller space at
$X\in \T_{g,n}$ are described in terms of the natural pairing
$ (q,\mu) \mapsto \int_{X} q\mu$ between the space $Q(X)$ and the
space $M(X)$ of $L^{\infty}$-measurable Beltrami differentials on $X$;
in particular, the tangent $T_{X} \T_{g,n}$ and cotangent
$T_{X}^{*} \T_{g,n}$ spaces are naturally isomorphic to
$M(X)/Q(X)^{\perp}$ and $Q(X)$, respectively.

The Teichm\"uller-Kobayashi metric on $\T_{g,n}$ is given by norm
duality on the tangent space $T_{X}\T_{g,n}$ from the norm
$||q||_{1} = \int_{X} |q|$ on the cotangent space $Q(X)$ at $X$. The
corresponding distance function is given by the formula
$d_{\T_{g,n}}(X,Y) = \inf \frac{1}{2} \log K(\phi)$ and measures the
minimal dilatation $K(\phi)$ of a quasiconformal map
$\phi: X \rightarrow Y$ respecting their markings.

We denote by $Q\T_{g,n} \cong T^{*}\T_{g,n}$ the complex vector-bundle
of holomorphic quadratic differentials over $\T_{g,n}$ and by
$Q_1\T_{g,n}$ the associated sphere-bundle of quadratic differentials
with unit mass. There is a natural norm-preserving action of
$\text{PSL}_2(\mathbb{R})$ on $Q\T_{g,n}$, with the diagonal matrices
giving the geodesic flow for the Teichm\"uller-Kobayashi metric of
$\T_{g,n}$. For each $(X,q)\in Q_1\T_{g,n}$, the orbit
$\text{PSL}_2(\mathbb{R}) \cdot (X,q) \subset Q_1\T_{g,n}$ induces a
holomorphic totally-geodesic isometry
$\mathbb{CH}^1 \cong
\text{SO}_2(\mathbb{R})\setminus\text{PSL}_2(\mathbb{R})
\hookrightarrow \T_{g,n}$,
which we refer to as the \textit{Teichm\"uller disk} generated by
$(X,q)$.~\\

\subsection*{Measured foliations}~\cite{FLP} 
Let $\mathcal{MF}_{g,n}$ denote the space of equivalent
classes\footnote{Two measured foliations $\mathcal{F},\mathcal{G}$ are
  equivalent $\mathcal{F}\thicksim \mathcal{G}$ if they differ by a
  finite sequence of Whitehead moves followed by an isotopy of
  $\Sigma_{g,n}$, preserving their transverse measures.~\cite{FLP}} of
nonzero (singular) measured foliations on $\Sigma_{g,n}$. It is known
that $\mathcal{MF}_{g,n}$ has the structure of a \textit{piecewise
  linear} manifold, which is homeomorphic to
$\mathbb{R}^{6g-6+2n}\setminus{\{0\}}$.

The geometric intersection number of a pair of measured foliations
$\mathcal{F},\mathcal{G}$, denoted by $i(\mathcal{F},\mathcal{G})$,
induces a continuous map
$i(\cdot,\cdot): \mathcal{MF}_{g,n} \times \mathcal{MF}_{g,n}
\rightarrow \mathbb{R}_{\geq 0}$,
which extends the geometric intersection pairing on the space of
(isotopy classes of) simple closed curves on
$\Sigma_{g,n}$.~\cite{Bonahon:currents}

Given $\mathcal{F} \in \mathcal{MF}_{g,n}$ and $X \in \T_{g,n}$, we
let $\lambda(\mathcal{F},X)$ denote the \textit{extremal length} of
$\mathcal{F}$ on the Riemann surface $X$ given by the formula
$\lambda(\mathcal{F},X)= \sup
\frac{\ell_{\rho}(\mathcal{F})^2}{\text{area}(\rho)}$,
where $\ell_{\rho} (\mathcal{F})$ denotes the $\rho$-length of
$\mathcal{F}$ and the supremum is over all (Borel-measurable)
conformal metrics $\rho$ of finite area on $X$.

Each nonzero quadratic differential $q \in Q(X)$ induces a conformal
metric $|q|$ on $X$, which is non-singular of zero curvature away from 
the zeros of $q$, and a measured foliation $\mathcal{F}(q)$ tangent to
vectors $v=v(z)\frac{\partial}{\partial z}$ with $q(v)=q(z)(v(z))^2 <0$. 
The transverse measure of the foliation $\mathcal{F}(q)$ is (locally)
given by integrating $|\text{Re}(\sqrt{q})|$ along arcs transverse to
its leaves.

We refer to $\mathcal{F}(q)$ as the vertical measured foliation
induced from $(X,q)$. In local coordinates, where $q=dz^2$ (such
coordinates exist away from the zeros of $q$), the metric $|q|$
coincides with the Euclidean metric $|dz|$ in the plane and the
measured foliation $\mathcal{F}(q)$ has leaves given by vertical lines
and transverse measure by the total horizontal variation
$|\text{Re}(dz)|$. We note that the measured foliation
$\mathcal{F}(-q)$ has (horizontal) leaves orthogonal to
$\mathcal{F}(q)$ and the product of their transverse measures is just
the area form of the conformal metric $|q|$ induced from $q$.

When it is clear from the context we often identify the measured
foliation $\mathcal{F}(q)$ with its equivalence class in
$\mathcal{MF}_{g,n}$. The following fundamental theorem will be used
in the next section.
\begin{thm}(\cite{Hubbard:Masur:fol};Hubbard-Masur)\label{thm:hubbard:masur}
  Let $X \in \T_{g,n}$; the map $q \mapsto \mathcal{F}(q)$ induces a
  homeomorphism $Q(X) \setminus \{0\} \cong \mathcal{MF}_{g,n}$.
  Moreover, $|q|$ is the unique extremal metric for $\mathcal{F}(q)$
  on $X$ and its extremal length is given by the formula
  $\lambda(\mathcal{F},X) = ||q||_1$.
\end{thm}

\addtocontents{toc}{\protect\setcounter{tocdepth}{1}}


\section{Extremal length geometry}\label{sec:balls}

Let $\mathbb{CH}^2\cong\{~~(z,w)~~ | ~~ |z|^2 + |w|^2 < 1~~ \} \subset
\C^2$ denote the complex hyperbolic plane, realized as the round unit
ball with its Kobayashi metric. In this section we will use measured
foliations and extremal length on Riemann surfaces to prove:
\begin{thm}\label{thm:balls}
  There is no holomorphic isometry $f: \mathbb{CH}^2 \hookrightarrow
  \T_{g,n}$ for the Kobayashi metric.
\end{thm}
\subsection*{Outline of the proof}
The proof leverages the fact that extremal length provides a link
between the geometry of Teichm\"uller geodesics and the geometric
intersection pairing for measured foliations.

By a theorem of Slodkowski
~\cite{Slodkowski:motions},~\cite{Earle:Markovic:isometries}, we
deduce that such an isometry would be totally-geodesic, it would send
real geodesics in $\mathbb{CH}^2$ to Teichm\"uller geodesics in
$\T_{g,n}$, preserving their length.  By
Theorem~\ref{thm:hubbard:masur}, we can parametrize the set of
Teichm\"uller geodesic rays from any base point $X\in\T_{g,n}$ by the
subspace of measured foliations $\mathcal{F}\in\mathcal{MF}_{g,n}$
with extremal length $\lambda(\mathcal{F},X)=1$.

Assuming the existence of $f$, we consider pairs of measured
foliations that parametrize orthogonal geodesic rays in the image of a
totally real geodesic hyperbolic plane $\mathbb{RH}^2\subset\mathbb{CH}^2$. 
We obtain a contradiction by computing their geometric intersection
number in two different ways.

On the one hand, we use the geometry of complex hyperbolic horocycles
and extremal length to show that the geometric intersection number
does not depend on the choice of the totally real geodesic plane. On
the other hand, by a direct geometric argument we show that this is
impossible. More precisely, we have:
\begin{prop}\label{prop:intersection}
  Let $q\in Q_1\T_{g,n}$ and $\mathcal{G}\in \mathcal{MF}_{g,n}$.
  There exist $v_1, \ldots, v_N \in \C^{*}$ such that
  $i(\mathcal{F}(e^{i\theta}q),\mathcal{G})=\sum_{i=1}^{N} |
  \text{Re}(e^{i\theta/2}v_i)|$ for all $\theta \in
  \mathbb{R}/2\pi\mathbb{Z}$.
\end{prop}
The proof of the proposition is given at the end of the section.\qed

See \S~\ref{sec:prelim} for background material in Teichm\"uller
theory and notation.

\subsection*{Complex hyperbolic horocycles}
Let $\gamma: [0,\infty) \rightarrow \mathbb{CH}^2$ be a geodesic ray
with unit speed. Since $\mathbb{CH}^2$ is a homogeneous space, we have
$\gamma=\alpha \circ \gamma_1$, where $\gamma_1(t) = (\tanh(t),0)$,
for $t\geq 0$, and $\alpha$ is a holomorphic isometry of
$\mathbb{CH}^2$. Each geodesic ray is contained in the image of unique
holomorphic totally-geodesic isometry
$\gamma: \mathbb{CH}^1 \hookrightarrow \mathbb{CH}^2$ satisfying
$\gamma(t)=\phi(\tanh(t))$; in particular, $\phi_1 (z) = (z,0)$, for
$z\in \Delta\cong \mathbb{CH}^1$. We note that every complex geodesic
$\phi: \mathbb{CH}^1 \hookrightarrow \mathbb{CH}^2$ arises uniquely
(up to pre-composition with an automorphism of $\mathbb{CH}^1$) as the
intersection of the unit ball in $\mathbb{C}^2$ with a complex affine
line.

Associated to each geodesic ray
$\gamma: [0,\infty) \rightarrow \mathbb{CH}^2$ is a pair of transverse
foliations of $\mathbb{CH}^2$, one by real geodesics asymptotic to
$\gamma$ and another by complex hyperbolic horocycles asymptotic to
$\gamma$. For each $p\in \mathbb{CH}^2$ there exists a \textit{unique}
geodesic $\gamma_p: \R \rightarrow \mathbb{CH}^2$ and a
\textit{unique} time $t_p\in \R$ such that $\gamma_p(t_p)=p$ and
$\displaystyle \lim_{t\rightarrow
  \infty}d_{\mathbb{CH}^2}(\gamma(t),\gamma_p(t))\rightarrow 0$.
For each $s\in \R_{+}$, we define the set
$H(\gamma,s)= \{~~ p \in \mathbb{CH}^2 ~~|~~ \exp(t_p)=s ~~\}$. The
collection of subsets $\{H(\gamma,s)\}_{s\in \mathbb{R}_{+}}$ defines
the foliation of $\mathbb{CH}^2$ by \textit{complex hyperbolic
horocycles} asymptotic to $\gamma$.

\subsection*{Extremal length horocycles}
Let $\gamma: [0,\infty) \rightarrow \T_{g,n}$ be a Teichm\"uller
geodesic ray with unit speed. It has a unique lift to
$\widetilde{\gamma}(t)=(X_t,q_t) \in Q_1\T_{g,n}$, such that
$\gamma(t)=X_t$ and $\widetilde{\gamma}(t) = \text{diag}(e^{t}, e^{-t})\cdot (X_0,q_0)$.
The map $q \mapsto (\mathcal{F}(q),\mathcal{F}(-q))$ gives an
embedding
$Q\T_{g,n} \hookrightarrow \mathcal{MF}_{g,n} \times
\mathcal{MF}_{g,n}$
which satisfies $||q||_1=i(\mathcal{F}(q),\mathcal{F}(-q))$ and sends
the lift $\widetilde{\gamma}(t)=(X_t,q_t)$ of Teichm\"uller geodesic
ray $\gamma$ to a path of the form
$(e^t\mathcal{F}(q),e^{-t}\mathcal{F}(-q))$.

The later description of a Teichm\"uller geodesic and
Theorem~\ref{thm:hubbard:masur} show that the extremal length of
$\mathcal{F}(q_t)$ along $\gamma$ satisfies
$\lambda(\mathcal{F}(q_t),X_s)=e^{2(t-s)}$ for all $t,s\in \R_{+}$,
which motivates the following definition. For each
$\mathcal{F} \in \mathcal{MF}_{g,n}$ the \textit{extremal length
  horocycles} asymptotic to $\mathcal{F}$ are the level-sets of
extremal length
$H(\mathcal{F},s) = \{~~ X \in \T_{g,n} ~~|~~ \lambda(\mathcal{F},X)=s
~~\}$
for $s\in \R_{+}$. The collection of subsets
$\{H(\mathcal{F},s)\}_{s\in \mathbb{R}_{+}}$ defines the foliation of
$\T_{g,n}$ by \textit{extremal length horocycles} asymptotic to
$\mathcal{F}$.

There is transverse foliation of $\T_{g,n}$ by real Teichm\"uller
geodesics with lifts $(X_t, q_t)$ that satisfy
$\mathcal{F}(q_{t}) \in \R_{+}\cdot \mathcal{F}$. One might expect
that this foliation of $\T_{g,n}$ is analogous to the foliation of
$\mathbb{CH}^2$ by geodesics that are positively asymptotic to
$\gamma$. Although this is not always true, it is true for
\textit{generic} measured foliations
$\mathcal{F} \in \mathcal{MF}_{g,n}$.
\begin{thm}\label{thm:masur:ue}(\cite{Masur:ergodic:geodesics};~H.~Masur)
  Let $(X_t,q_t)$ and $(Y_t,p_t)$ be two Teichm\"uller geodesics and
  $\mathcal{F}(q_0)\in\mathcal{MF}_{g,n}$ be uniquely
  ergodic.~\footnote{A measured foliation $\mathcal{F}$ is
    \textit{uniquely ergodic} if it is minimal and admits a unique, up
    to scaling, transverse measure; in particular,
    $i(\gamma,\mathcal{F}) > 0 $ for all simple closed curves
    $\gamma$. Compare with ~\cite{Masur:ergodic:geodesics}.} Then
  $lim_{t\rightarrow\infty}d_{\T_{g,n}}(X_t,Y_t)\rightarrow 0$ if and
  only if $\mathcal{F}(q_0)=\mathcal{F}(p_0)$ in $\mathcal{MF}_{g,n}$
  and $\lambda(\mathcal{F}(q_0),X_0)=\lambda(\mathcal{F}(p_0),Y_0)$.
\end{thm}
\begin{remark}
  It is known that this result is not true for measured foliations
  that are not uniquely ergodic.
\end{remark}
\subsection*{Proof of Theorem~\ref{thm:balls}}
Let $f: \mathbb{CH}^2 \hookrightarrow \T_{g,n}$ be a holomorphic
isometry for the Kobayashi metric. We summarize the proof in the
following three steps:~\\

\noindent\textbf{1.} \textit{Asymptotic behavior of geodesics
  determines the extremal length horocycles.}

\noindent\textbf{2.} \textit{The geometry of horocycles determines the
  geometric intersection pairing.}

\noindent\textbf{3.} \textit{Get a contradiction by a direct
  computation of the geometric intersection pairing.}~\\

\noindent\textbf{Step 1.} Let $X =f((0,0)) \in \T_{g,n}$ and $q,p \in Q_1(X)$
unit area quadratic differentials generating the two Teichm\"uller
geodesic rays $f(\gamma_1)$,$f(\gamma_2)$, where $\gamma_1$,$\gamma_2$
are two orthogonal geodesic rays in $\mathbb{CH}^2$ contained in the
image of the totally real geodesic hyperbolic plane
$\mathbb{RH}^2\subset\mathbb{CH}^2$; explicitly, they are given by the
formulas $\gamma_1(t) = (\tanh(t),0)$, $\gamma_2(t)=(0,\tanh(t))$, for
$t\geq0$.

For every $(X,q) \in Q_1\T_{g,n}$ there is a dense set of
$\theta \in \mathbb{R}/2\pi\mathbb{Z}$ such that the measured
foliation $\mathcal{F}(e^{i\theta}q)$ is uniquely
ergodic~\cite{Chaika:Cheung:Masur:winning}; hence, we can assume
without loss of generality (up to a holomorphic automorphism of
$\mathbb{CH}^2$) that both $\mathcal{F}(q)$ and $\mathcal{F}(p)$ are
(minimal) uniquely ergodic measured foliations. In particular, we can
apply Theorem~\ref{thm:masur:ue} to study the extremal length
horocycles asymptotic to $\mathcal{F}(q)$ and $\mathcal{F}(p)$
respectively.

The complex hyperbolic horocycle $H(\gamma_1,1)$ is characterized by
the property that for the points $P \in H(\gamma_1,1)$ the geodesic
distance between $\gamma_P(t)$ and $\gamma_1(t)$ tends to zero as
$t \rightarrow +\infty$, where $\gamma_P(t)$ is the unique geodesic
with unit speed through $P$ that is positively asymptotic to
$\gamma_1$. Applying Theorem~\ref{thm:masur:ue} we conclude that:
\begin{equation}\label{eq:1}
f(\mathbb{CH}^2) \cap H(\mathcal{F}(q),1) = f( H(\gamma_1,1))
\end{equation}
\begin{equation}\label{eq:2}
f(\mathbb{CH}^2) \cap H(\mathcal{F}(p),1) = f( H(\gamma_2,1))  
\end{equation}

\noindent\textbf{Step 2.}  Let $\delta$ be the (unique) complete real 
geodesic in $\mathbb{CH}^2$, which is asymptotic to $\gamma_1$ in the
positive direction and to $\gamma_2$ in the negative direction,
i.e. its two endpoints are $(1,0),(0,1) \in \C^2$ in the boundary of
the unit ball. Let $P_1$ and $P_2$ be the two points where $\delta$
intersects the horocycles $H(\gamma_1,1)$ and $H(\gamma_2,1)$,
respectively. See~\ref{fig1}.

The image of $\delta$ under the map $f$ is a Teichm\"uller geodesic
which is parametrized by a pair of measured foliations
$\mathcal{F},\mathcal{G} \in \mathcal{MF}_{g,n}$ with
$i(\mathcal{F},\mathcal{G})=1$ and its unique lift to $Q_1\T_{g,n}$ is
given by $(e^t \mathcal{F},e^{-t}\mathcal{G})$, for $t\in \R$. Let
$\widetilde{P_i} = (e^{t_i} \mathcal{F},e^{-t_i}\mathcal{G})$, for
$i=1,2$, denote the lifts of $P_1,P_2$ along the geodesic
$\delta$. Then, the distance between the two points is given by
$d_{\mathbb{CH}^2}(P_1,P_2) = t_2-t_1$. From Step 1, we conclude that
$e^{t_1}\mathcal{F} = \mathcal{F}(q)$ (\ref{eq:1}) and
$e^{-t_2}\mathcal{G}=\mathcal{F}(p)$ ((\ref{eq:2}). Therefore we have
$i(\mathcal{F}(q),\mathcal{F}(p)) = e^{t_1 - t_2}$.
\begin{figure}[ht]
	\centering
  \includegraphics[scale=0.3]{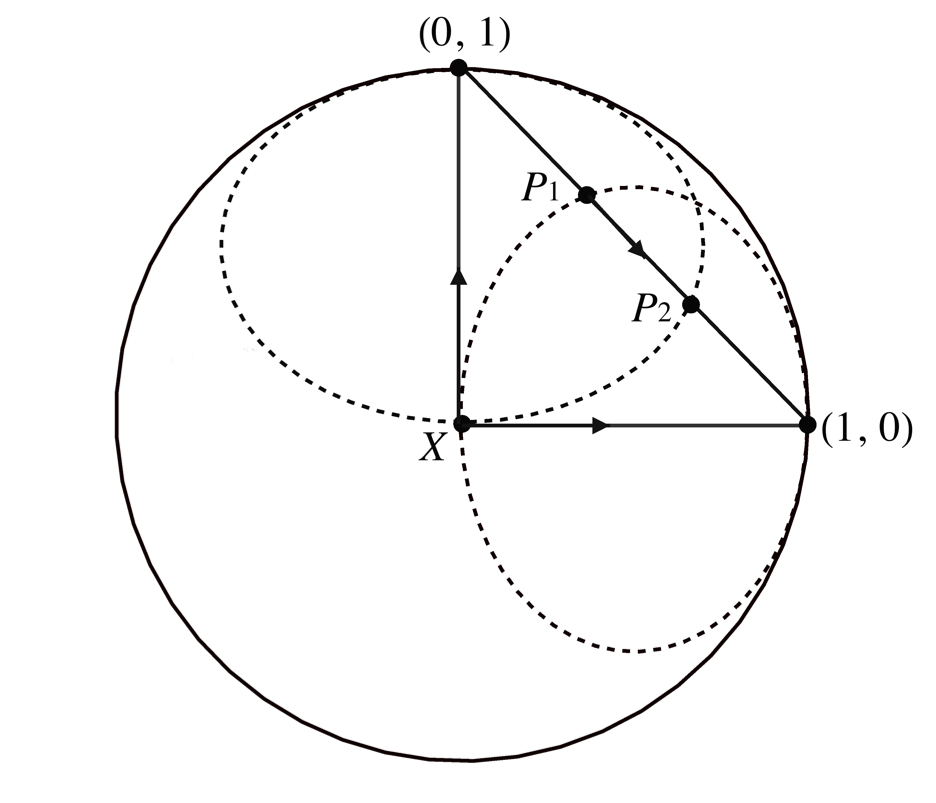}
  \caption{The real slice of $\mathbb{CH}^2\subset \mathbb{C}^2$
    coincides with the Klein model $\mathbb{RH}^2\subset \mathbb{R}^2$
    of the real hyperbolic plane of constant curvature $-1$. }
	\label{fig1}
\end{figure}
\begin{remark}
  A simple calculation shows that $t_2 -t_1= \log(2)$; hence,
  $i(\mathcal{F}(q),\mathcal{F}(p))= \frac{1}{2}$.
\end{remark}

\noindent\textbf{Step 3.} 
The holomorphic automorphism given by
$\phi (z,w)= (e^{-i \theta}z,w)$, for $(z,w)\in \mathbb{CH}^2$, is an
isometry of $\mathbb{CH}^2$ and sends the two horocycles
$H(\gamma_i,1)$ to the horocycles $H(\phi(\gamma_i),1)$, for
$i=1,2$. The Teichm\"uller geodesic ray $f(\phi(\gamma_1))$ is now
generated by $e^{i \theta} q$, whereas the Teichm\"uller geodesic ray
$f(\phi(\gamma_2))$ is still generated by $p \in Q(X)$. Since the
distance between $P_1$ and $P_2$ is equal to the distance between
$\phi(P_1)$ and $\phi(P_2)$, using Step 2 and the continuity of the
geometric intersection pairing we conclude that
$i(\mathcal{F}(e^{i \theta}q), \mathcal{G}) = \frac{1}{2}$ for all
$\theta \in \mathbb{R}/2\pi\mathbb{Z}$. However, this contradicts the
following Proposition~\ref{prop:intersection}.\qed
\restate[Proposition]{prop:intersection}{
  Let $q\in Q_1\T_{g,n}$ and $\mathcal{G}\in \mathcal{MF}_{g,n}$.
  There exist $v_1, \ldots, v_N \in \C^{*}$ such that
  $i(\mathcal{F}(e^{i\theta}q),\mathcal{G})=\sum_{i=1}^{N} |
  \text{Re}(e^{i\theta/2}v_i)|$ for all $\theta \in
  \mathbb{R}/2\pi\mathbb{Z}$.}
\begin{proof}[Proof of Proposition~\ref{prop:intersection}]
  Let $q \in Q(X)$ be a unit area quadratic differential. We assume
  first that $q$ has no poles and that $\mathcal{G}$ is an isotopy
  class of simple closed curves. The metric given by $|q|$ is flat
  with conical singularities of negative curvature at its set of zeros
  and hence the isotopy class of simple closed curves $\mathcal{G}$
  has a unique geodesic representative, which is a finite union of
  saddle connections of $q$. In particular, we can readily compute
  $i(\mathcal{F}(e^{i\theta}q),\mathcal{G})$ by integrating
  $|\text{Re} (\sqrt{e^{i\theta}q} )|$ along the union of these saddle
  connections. It follows that:
  \begin{equation}\label{eq:intersection}
    i(\mathcal{F}(e^{i\theta}q),\mathcal{G}) = \sum_{i=1}^{N} |
    \text{Re}(e^{i\theta/2}v_i)|
    \quad \text{for all} \quad \theta \in \mathbb{R}/2\pi\mathbb{Z}
  \end{equation}
  where $N$ denotes the number of the saddles connections and
  $\{ v_i \}_{i=1}^{N} \subset \C^{*}$ are their associated holonomy
  vectors.

  We note that when $q$ has simple poles, there need not be a geodesic
  representative in $\mathcal{G}$ anymore. Nevertheless, equation
  (\ref{eq:intersection}) is still true by applying the argument to a
  sequence of length minimizing representatives.

  Finally, we observe that the number of saddle connections $N$ is
  bounded from above by a constant that depends only on the topology
  of the surface. Combining this observation with the fact that any
  $\mathcal{G} \in \mathcal{MF}_{g,n}$ is a limit of simple closed
  curves and that the geometric intersection pairing $i(\cdot,\cdot):
  \mathcal{MF}_{g,n} \times \mathcal{MF}_{g,n} \rightarrow \R$ is
  continuous, we conclude that equation~(\ref{eq:intersection}) is
  true in general.
\end{proof}



\addtocontents{toc}{\protect\setcounter{tocdepth}{1}}

\section{Symmetric spaces vs Teichm\"uller spaces}\label{sec:bsds}
Let $\T_{g,n}\subset\C^{3g-3+n}$ be a Teichm\"uller space and
$\mathcal{B}\subset\C^N$ a bounded symmetric domain equipped with their
\textit{Kobayashi} metrics. In this section, we complete the proof of
the following theorem.
\begin{thm}\label{thm:bsds}
  Let $\mathcal{B}\subset\C^{N}$ be a bounded symmetric domain and $\T_{g,n}$ be
  a Teichm\"uller space with $\text{dim}_{\C}\mathcal{B},\text{dim}_{\C}\T_{g,n}
  \geq 2$. There are no holomorphic isometric immersions \[\mathcal{B}
  \xhookrightarrow{~~~f~~~} \T_{g,n} \quad \text{or} \quad \T_{g,n}
  \xhookrightarrow{~~~f~~~} \mathcal{B}\] such that $df$ is an isometry for the
  Kobayashi norms on tangent spaces.~\\
\end{thm}
\begin{remarks}~\\
  1. Torelli maps (associating to a marked Riemann surface the
  Jacobians of its finite covers) give rise to holomorphic maps
  $\T_{g,n} \xhookrightarrow{~~~\mathcal{T}~~~} \mathcal{H}_h$ into
  bounded symmetric domains (Siegel spaces). It is known that these
  maps are not isometric for the Kobayashi metric in most
  directions.~\cite{McMullen:covs}~\\
  2. For a similar result about holomorphic isometric submersions
  see~\cite{Antonakoudis:birational}.
\end{remarks}
\subsection*{Outline of the proofs}
The proof that $\mathcal{B} \not\hookrightarrow \T_{g,n}$ follows from
Theorem~\ref{thm:balls} (rank one) and a classical application of
Sullivan's rigidity theorem (higher rank). The new ingredient we
introduce in this section is a comparison of the \textit{roughness} of
Kobayashi metric for bounded symmetric domains and Teichm\"uller
spaces, which we will use to prove that $\T_{g,n} \not\hookrightarrow \mathcal{B}$ 
\qed ~\\

\subsection*{Preliminaries on symmetric spaces}~\\
We give a quick review of the main features of symmetric spaces, from
a complex analysis perspective, which we use in the proof. We refer
to~\cite{Helgason:book:dglgss},~\cite{Satake:book:symmetric} for more
details.

Let $\mathcal{B} \subset \mathbb{C}^N$ be a bounded symmetric domain
and $p\in \mathcal{B}$. There is a \textit{unique}, up to
post-composition with a linear map, holomorphic embedding
$\mathcal{B}\xhookrightarrow{~~i~~}\C^N$ such that
$i(\mathcal{B}) \subset \C^N$ is a \textit{strictly convex} circular
domain with $i(p)= 0 \in \C^N$, which we refer to as the
Harish-Chandra realization of $\mathcal{B}$ centered at
$p\in \mathcal{B}$.

It is known that the Harish-Chandra realization of
$\mathcal{B}\subset \C^N$ has the following useful description. There
is a finite dimensional linear subspace
$V_{\mathcal{B}} \subset M_{n,m}(\C)$, of the space of complex
$n \times m$ matrices, such that
$\mathcal{B}\cong \{~~V \in V_{\mathcal{B}} ~~|~~ ||V||_{\mathcal{B}}
< 1~~ \}$
is the unit ball for the operator norm on $V_{\mathcal{B}}$, where
$||V||_{\mathcal{B}} = \text{sup}_{||\xi||_2 =1} ||V(\xi)||_2$, for
$V \in M_{n,m}(\C)$. We note that there is a natural identification
$T_{p}\mathcal{B} \cong V_{\mathcal{B}}\cong \C^N$.~\cite{Satake:book:symmetric}

The Kobayashi norm on $T_{p}\mathcal{B}\cong V_{\mathcal{B}}$
coincides with the operator norm $||V||_\mathcal{B}$, for
$V\in V_{\mathcal{B}} \subset M_{n,m}(\C)$ and the Kobayashi distance
from the origin is given by the formula
$d_\mathcal{B}(0,V)=\frac{1}{2}\log(\frac{1+||V||_\mathcal{B}}{1-||V||_\mathcal{B}})$,
for $V \in \mathcal{B}$.~\cite{Kubota:sym}~\\

\subsection*{Roughness of the Kobayashi metric}~\\
The following proposition describes the roughness of the Kobayashi
distance for bounded symmetric domains.

\begin{prop}\label{prop:bsd}
  Let $V: (-1,1) \rightarrow \mathcal{B}$ be a real-analytic path with
  $V(0)\neq p$. There is an integer $K >0$ and an $\epsilon >0 $ such
  that  $d_{\mathcal{B}}(p,V(\cdot)): [0,\epsilon)\rightarrow \mathcal{B}$
  is a real-analytic function of $t^{1/K}$ for $t \in [0,\epsilon)$.
\end{prop}
\begin{proof}
  Let
  $\mathcal{B} = \{~~||V||_{\mathcal{B}} < 1~~\}\subset
  V_{\mathcal{B}} \subset M_{n,m}(\C)$
  be the Harish-Chandra realization of $\mathcal{B}$ centered at $p$.
  For each $t\in (-1,1)$, we denote by $\lambda_i(t)$, for
  $i=1, \ldots, n$, the eigenvalues of the (positive) square matrix
  $V(t)^{*}V(t)$, counted with multiplicities, where $V^{*}$ denotes
  the Hermitian adjoint of $V$.

  The eigenvalues of $V(t)^{*}V(t)$ are the zeros of a polynomial, the
  coefficients of which are real-analytic functions of $t\in (-1,1)$.
  Therefore, the points $(t,\lambda_i (t)) \in \C^2$ for
  $i=1, \ldots, n$ and $t\in (-1,1)$ are contained in an algebraic
  curve $C= \{~~(t,\lambda) \in \C^2 ~~|~~ P(t,\lambda)=0~~\}$, which
  is equipped with a finite-degree branched covering map to
  $\mathbb{C}$ given by $(t,\lambda) \mapsto t$, for  $(t,\lambda) \in C$.

  Since the operator norm is given by the formula
  $|| V(t) ||_\mathcal{B} = \sup \{| \lambda_i (t)|^{1/2}\}_{i=1}^n$,
  the proof of the proposition follows by considering the Puiseux
  series expansion for $\lambda_i(t)$'s and the formula
  $d_\mathcal{B}(0,V(t))=\frac{1}{2}\log(\frac{1+||V(t)||_\mathcal{B}}{1-||V(t)||_\mathcal{B}})$.
\end{proof}

The roughness of the Kobayashi metric for Teichm\"uller spaces is
described by the following two theorems of M.~Rees.
\begin{thm}(\cite{Rees:distance:c2}; M. Rees)\label{thm:c2} 
  The Teichm\"uller distance
  $d_{\T_{g,n}} : \T_{g,n}\times\T_{g,n} \rightarrow \R_{\geq 0}$ is
  $C^2$-smooth on the complement of the diagonal
  $d_{\T_{g,n}}^{-1}(0)$.
\end{thm}
\begin{thm}(\cite{Rees:distance:notc2}; M. Rees)\label{thm:notc2}
  When $\text{dim}_{\C}\T_{g,n} \geq 2$, the Teichm\"uller distance
  $d_{\T_{g,n}} : \T_{g,n}\times\T_{g,n} \rightarrow \R_{\geq 0}$ is
  \textit{not} $C^{2+\epsilon}$ for any $\epsilon>0$.

  Moreover, let $X,Y \in \T_{g,n}$ be two distinct points connected by
  a (real) Teichm\"uller geodesic which is generated by a quadratic
  differential $q\in Q_1(X)$, with either a zero of order two or
  number of poles less than $n$. There is a real analytic path
  $X(t):(-1,1)\rightarrow \T_{g,n}$ with $X(0)=X$ such that the
  distance $d_{\T_{g,n}}(X(t),Y)$ is not $C^{2+h}$-smooth at $t=0$,
  for every gauge function $h(t)$ with $\lim_{t \rightarrow 0}
  \frac{h(t)}{1/log(1/|t|)} = 0$.~\\
\end{thm}

\subsection*{Proof of Theorem~\ref{thm:bsds}}~\\

Let $\mathcal{B}\subset\C^{N}$ be a bounded symmetric domain and
$\T_{g,n}$ a Teichm\"uller space with
$\text{dim}_{\C}\mathcal{B},\text{dim}_{\C}\T_{g,n} \geq 2$. Using the
fact that bounded symmetric domains and Teichm\"uller spaces contain
holomorphic isometric copies of $\mathbb{CH}^1$ through every point
and complex direction, and a theorem of
Slodkowski~\cite{Slodkowski:motions},~\cite{Earle:Markovic:isometries},
we deduce that any holomophic map $f$ between $\mathcal{B}$ and
$\T_{g,n}$ which is an isometry for the Kobayashi norms on tangent
spaces would be totally-geodesic and would therefore preserve the
Kobayashi distance for pairs of points.~\\

\noindent\textbf{($\mathcal{B} \not\hookrightarrow \T_{g,n}$)}~\\

Theorem~\ref{thm:balls} shows that there is no holomorphic isometry
$f: \mathbb{CH}^2 \rightarrow \T_{g,n}$. Moreover, an application of
Sullivan's rigidity theorem (see ~\cite{Tanigawa:holomap} for a
precise statement) shows that there is no proper holomorphic map
$f: \mathbb{CH}^1 \times \mathbb{CH}^1 \rightarrow \T_{g,n}$, hence
neither is such a holomorphic map that is an isometry.

However, for every bounded symmetric domain $\mathcal{B}$ with
$\text{dim}_{\C}\mathcal{B} \geq 2$ there is either a holomorphic
totally-geodesic isometry $\mathbb{CH}^2 \hookrightarrow \mathcal{B}$
(rank one) or a holomorphic totally-geodesic isometry
$\mathbb{CH}^1\times \mathbb{CH}^1 \hookrightarrow \mathcal{B}$
(higher rank).~\cite{Kobayashi:book:metric} We conclude that there is
no holomorphic isometric immersion $f: \mathcal{B} \hookrightarrow \T_{g,n}$.~\\

\noindent\textbf{($\T_{g,n} \not\hookrightarrow \mathcal{B}$)}~\\

Let $f: \T_{g,n} \hookrightarrow \mathcal{B}$ be a holomorphic
isometric immersion. Since $\text{dim}_{\mathbb{C}}\T_{g,n} \geq 2$,
we can choose two distinct points $X,Y \in \T_{g,n}$ as described in
Theorem~\ref{thm:notc2}; hence there is a real analytic path
$X(t):(-1,1)\rightarrow \T_{g,n}$ with $X(0)=X$ such that the
Teichm\"uller distance $d_{\T_{g,n}}(X(t),Y)$ is not $C^{2+h}$-smooth
at $t=0$ for every gauge function $h(t)$ with
$\lim_{t \rightarrow 0} \frac{h(t)}{1/log(1/|t|)} = 0$.

Let $p=f(Y)\in \mathcal{B}$ and
$V(\cdot): (-1,1) \rightarrow \mathcal{B}$ be the real analytic path
given by $V(t) = f(X(t))$ for $t \in (-1,1)$. Theorem~\ref{thm:c2}
shows $d_{\mathcal{B}}(p,V(t))$ is $C^2$-smooth at $t=0$ and
Proposition~\ref{prop:bsd} shows that it is real analytic in
$t^{1/K}$, for some fixed integer $K>0$, for all sufficiently small
$t \geq 0$. Therefore, it follows that $d_{\T_{g,n}}(X(t),Y)$ is
$C^{2+\frac{1}{K}}$-smooth, but this contradicts the choice of the
path $X(t) \in \T_{g,n}$, given by Theorem~\ref{thm:notc2}, by
considering the gauge function $h(t)= t^{1/K}$ for $t \geq 0$. We
conclude that there is no holomorphic isometric immersion
$f: \T_{g,n} \hookrightarrow \mathcal{B}$.\qed




\end{document}